\documentclass[12pt]{amsart}
\usepackage{amsmath,amssymb,amsbsy,amsfonts,amsthm,latexsym,amsopn,amstext,
                            amsxtra,euscript,amscd}

\newfont{\teneufm}{eufm10}
\newfont{\seveneufm}{eufm7}
\newfont{\fiveeufm}{eufm5}
\newfam\eufmfam
                 \textfont\eufmfam=\teneufm \scriptfont\eufmfam=\seveneufm
                 \scriptscriptfont\eufmfam=\fiveeufm
\def\bbbc{{\mathchoice {\setbox0=\hbox{$\displaystyle\rm C$}\hbox{\hbox
to0pt{\kern0.4\wd0\vrule height0.9\ht0\hss}\box0}}
{\setbox0=\hbox{$\textstyle\rm C$}\hbox{\hbox
to0pt{\kern0.4\wd0\vrule height0.9\ht0\hss}\box0}}
{\setbox0=\hbox{$\scriptstyle\rm C$}\hbox{\hbox
to0pt{\kern0.4\wd0\vrule height0.9\ht0\hss}\box0}}
{\setbox0=\hbox{$\scriptscriptstyle\rm C$}\hbox{\hbox
to0pt{\kern0.4\wd0\vrule height0.9\ht0\hss}\box0}}}}
\def\bbbq{{\mathchoice {\setbox0=\hbox{$\displaystyle\rm
Q$}\hbox{\raise
0.15\ht0\hbox to0pt{\kern0.4\wd0\vrule height0.8\ht0\hss}\box0}}
{\setbox0=\hbox{$\textstyle\rm Q$}\hbox{\raise
0.15\ht0\hbox to0pt{\kern0.4\wd0\vrule height0.8\ht0\hss}\box0}}
{\setbox0=\hbox{$\scriptstyle\rm Q$}\hbox{\raise
0.15\ht0\hbox to0pt{\kern0.4\wd0\vrule height0.7\ht0\hss}\box0}}
{\setbox0=\hbox{$\scriptscriptstyle\rm Q$}\hbox{\raise
0.15\ht0\hbox to0pt{\kern0.4\wd0\vrule height0.7\ht0\hss}\box0}}}}
\def\bbbt{{\mathchoice {\setbox0=\hbox{$\displaystyle\rm
T$}\hbox{\hbox to0pt{\kern0.3\wd0\vrule height0.9\ht0\hss}\box0}}
{\setbox0=\hbox{$\textstyle\rm T$}\hbox{\hbox
to0pt{\kern0.3\wd0\vrule height0.9\ht0\hss}\box0}}
{\setbox0=\hbox{$\scriptstyle\rm T$}\hbox{\hbox
to0pt{\kern0.3\wd0\vrule height0.9\ht0\hss}\box0}}
{\setbox0=\hbox{$\scriptscriptstyle\rm T$}\hbox{\hbox
to0pt{\kern0.3\wd0\vrule height0.9\ht0\hss}\box0}}}}
\def\bbbs{{\mathchoice
{\setbox0=\hbox{$\displaystyle     \rm S$}\hbox{\raise0.5\ht0\hbox
to0pt{\kern0.35\wd0\vrule height0.45\ht0\hss}\hbox
to0pt{\kern0.55\wd0\vrule height0.5\ht0\hss}\box0}}
{\setbox0=\hbox{$\textstyle        \rm S$}\hbox{\raise0.5\ht0\hbox
to0pt{\kern0.35\wd0\vrule height0.45\ht0\hss}\hbox
to0pt{\kern0.55\wd0\vrule height0.5\ht0\hss}\box0}}
{\setbox0=\hbox{$\scriptstyle      \rm S$}\hbox{\raise0.5\ht0\hbox
to0pt{\kern0.35\wd0\vrule height0.45\ht0\hss}\raise0.05\ht0\hbox
to0pt{\kern0.5\wd0\vrule height0.45\ht0\hss}\box0}}
{\setbox0=\hbox{$\scriptscriptstyle\rm S$}\hbox{\raise0.5\ht0\hbox
to0pt{\kern0.4\wd0\vrule height0.45\ht0\hss}\raise0.05\ht0\hbox
to0pt{\kern0.55\wd0\vrule height0.45\ht0\hss}\box0}}}}
\def\bbbz{{\mathchoice {\hbox{$\sf\textstyle Z\kern-0.4em Z$}}
{\hbox{$\sf\textstyle Z\kern-0.4em Z$}}
{\hbox{$\sf\scriptstyle Z\kern-0.3em Z$}}
{\hbox{$\sf\scriptscriptstyle Z\kern-0.2em Z$}}}}

 \newtheorem{thm}{Theorem}
 
 \newtheorem{lem}[thm]{Lemma}
 
 \theoremstyle{definition}
 
 \theoremstyle{remark}

\def\cA{{\mathcal A}}
\def\cB{{\mathcal B}}

\def\cQ{{\mathcal Q}}
\def\cR{{\mathcal R}}
\def\cS{{\mathcal S}}

\def\cU{{\mathcal U}}
\def\cV{{\mathcal V}}
\def\cW{{\mathcal W}}

\def\({\left(}
\def\){\right)}
\def\[{\left[}
\def\]{\right]}
\def\<{\langle}
\def\>{\rangle}

\def\fl#1{\left\lfloor#1\right\rfloor}
\def\rf#1{\left\lceil#1\right\rceil}

\def\F{\mathbb{F}}

\def\mand{\qquad\mbox{and}\qquad}

\begin{document}

\title[Counting  Additive Decompositions of Quadratic Residues]{Counting Additive Decompositions of Quadratic Residues in Finite Fields}

\author[S. R. Blackburn]{Simon R. Blackburn}  
\address{Department of Mathematics, Royal Holloway University of London, 
Egham, Surrey, TW20 0EX, UK}  
\email{s.blackburn@rhul.ac.uk}

\author[S. V. Konyagin]{Sergei V.~Konyagin}
\address{Steklov Mathematical Institute,
8, Gubkin Street, Moscow, 119991, Russia}
 \email{konyagin@mi.ras.ru}

\author[I. E.~Shparlinski]{Igor E.~Shparlinski}
\address{Department of Pure Mathematics, University of 
New South Wales, Sydney, NSW 2052 Australia}
\email{igor.shparlinski@unsw.edu.au}

\begin{abstract} 
  We say that a set $\cS$ is additively decomposed into two sets $\cA$
  and $\cB$ if $\cS = \{a+b~:~a\in \cA, \ b \in \cB\}$.
  A.~S{\'a}rk{\"o}zy has recently conjectured that the set $\cQ$ of
  quadratic residues modulo a prime $p$ does not have nontrivial
  decompositions.  Although various partial results towards this
  conjecture have been obtained, it is still open.  Here we obtain a
  nontrivial upper bound on the number of such decompositions.
\end{abstract}



\maketitle

\section{Introduction}

Given two subsets $\cA, \cB \subseteq \F_q$ 
of the finite field $\F_q$ of $q$ elements, we
define their sum as 
$$
\cA + \cB  = \{a+b~:~a\in \cA, \ b \in \cB\}.
$$
A set  $\cS \subseteq \F_q$  is called {\it additively decomposable\/}
into two sets if $\cS = \cA + \cB$ for some sets $\cA, \cB$ with 
$$
\min\{\#\cA, \# \cB\} \ge 2.
$$

S{\'a}rk{\"o}zy~\cite{Sark} has conjectured that the set $\cQ$ of quadratic 
residues modulo a prime $p$ does not have additive decompositions and shown towards 
this conjecture that any additive decomposition
$$
\cQ = \cA + \cB
$$
satisfies 
$$
\frac{p^{1/2}}{3\log p} \le \min\{\#\cA, \# \cB\} \le \max\{\#\cA, \# \cB\} \le p^{1/2} \log p .
$$
The method also works for an arbitrary finite field of odd characteristic. 
In~\cite{Shp} this result has been improved to 
\begin{equation}
\label{eq:low up}
c q^{1/2}  \le \min\{\#\cA, \# \cB\} \le \max\{\#\cA, \# \cB\} \le C q^{1/2}, 
\end{equation}
for some absolute constants $C\ge c > 0$ 
(and also generalised to other multiplicative subgroups of $\F_q^*$).

Shkredov~\cite{Shkr} has recently made remarkable
progress towards the conjecture of S{\'a}rk{\"o}zy~\cite{Sark}
by showing that the conjecture holds with $\cA=\cB$. That is,
$\cQ \ne \cA + \cA$ for any set $\cA \subseteq \F_p$.

Furthermore, Dartyge and S{\'a}rk{\"o}zy~\cite{DaSa} have made a 
similar conjecture for the set $\cR$ of primitive roots modulo $p$.
We also refer 
to~\cite{DaSa,Els,Sark} for further references about set decompositions.

For an odd prime power $q$ we denote by $N(q)$ the 
total number of pairs $(\cA,\cB)$ of sets $\cA, \cB \subseteq \F_q$ 
that provide an additive decomposition of the set of quadratic residues
of $\F_q$, that is,  the set $\cQ = \{x^2~:~x\in \F_q^*\}$. 
The conjecture of S{\'a}rk{\"o}zy~\cite{Sark}  is equivalent to 
the statement that $N(q) = 0$ when $q$ is an odd prime (and is probably
true for any odd prime power as well). 

The bound~\eqref{eq:low up} implies 
$$
N(q) \le \exp\(O(q^{1/2} \log q)\). 
$$

Here we obtain a more precise estimate:

\begin{thm}
\label{thm:Nq} 
For any odd prime power $q$, we have
$$
N(q) \le \exp\(O(q^{1/2})\). 
$$
\end{thm}

Finally, we remark that the argument we use to prove
Theorem~\ref{thm:Nq} can be 
extended to prove results on additive decompositions of many other ``multiplicatively''
defined sets, such as cosets of multiplicative groups 
and sets of primitive elements  of $\F_q^*$. See~\cite{DaSa,Shp} 
for analogues of~\eqref{eq:low up}  for such sets. 


\section{Bounds of Multiplicative Character Sums}

As usual, we use the expressions $A \ll B$ and $A=O(B)$ to
mean $|A|\leq cB$ for some constant $c$.

We recall the following bound on a double character sum
due to Karatsuba~\cite{Kar1}, see also~\cite[Chapter~VIII, Problem~9]{Kar2}, 
which can easily be derived  from the
Weil bound (see~\cite[Corollary~11.24]{IwKow}) and the H{\"o}lder inequality.

\begin{lem}
\label{lem:DoubleSums} For any integer $\nu \ge 1$, sets
$\cU, \cV \subseteq \F_q$ and nontrivial multiplicative character  $\chi$ 
of $\F_q$,  
we have
$$
\sum_{u\in \cU}\sum_{v \in \cV} \chi(u+v)  \ll  (\# \cU)^{1-1/2\nu} \#\cV q^{1/4\nu} +  
(\# \cU)^{1-1/2\nu} (\#\cV)^{1/2} q^{1/2\nu} ,
$$
where the implied constant depends only on $\nu$. 
\end{lem}


We obtain the following result as a corollary of Lemma~\ref{lem:DoubleSums}:

\begin{lem}
\label{lem:Char 1}  
For any $\varepsilon>0$ if for
two sets $\cU, \cV \subseteq \F_q$ with $\# \cV \ge q^{\varepsilon}$ and a
nontrivial multiplicative character  $\chi$
of $\F_q$,  we have
 $\chi(u+v) = 1$ for all pairs $(u,v) \in \cU \times \cV$, then
$\# \cU \ll q^{1/2}$ where the implied constant depends only on $\varepsilon$.
\end{lem}

\begin{proof} We see from Lemma~\ref{lem:DoubleSums} that 
\begin{equation*}
\begin{split}
\# \cU \# \cV & = \sum_{u\in \cU}\sum_{v \in \cV} \chi(u+v) \\
&  \ll 
 (\# \cU)^{1-1/2\nu} \#\cV q^{1/4\nu} +  
(\# \cU)^{1-1/2\nu} (\#\cV)^{1/2} q^{1/2\nu}.
 \end{split}
\end{equation*}
Taking $\nu$ sufficiently large so that the first term dominates
(for example, taking $\nu = \rf{(2\varepsilon)^{-1}}$ so that $\#\cV \ge q^{1/2\nu}$)
we find that
$$
\# \cU \# \cV  \ll  (\# \cU)^{1-1/2\nu} \#\cV q^{1/4\nu}, 
$$
which implies the result. 
\end{proof}

We remark that the bounds~\eqref{eq:low up} follow from S\'ark\"ozy's
result~\cite{Sark} and Lemma~\ref{lem:Char 1}. To see this, note that the upper
bound follows by taking $\chi$ to be the quadratic character in
Lemma~\ref{lem:Char 1}, and taking $\cU=\cA$ and $\cV=\cB$ (and then
$\cU=\cB$ and $\cV=\cA$). The lower bound now follows since $\# \cQ\leq
\# \cA \# \cB$.

\section{Proof of Theorem~\ref{thm:Nq}}
\label{sec:Proof}

The proof  of Theorem~\ref{thm:Nq}  is instant from the 
following result, which is of independent interest.

For positive integers $k$ and $m$, let $N(k,m,q)$  denote 
the  number of pairs $(\cA,\cB)$ of sets $\cA, \cB \subseteq \F_q$ 
with $\# \cA = k$, $\cB=m$ such that $\cQ = \cA + \cB$. 

To simplify formulas we extend the definition of binomial coefficients to
all non-negative real numbers. More precisely, for a real 
$z \ge 0$ and an integer $n$  we set
$$
\binom{z}{n} = \binom{\fl{z}}{n}.
$$

\begin{lem}
\label{lem:Nkmq} For any fixed $\varepsilon > 0$ there is a constant $c>0$ 
such that for all integers  
$k$ and $m$ with $q> k> q^\varepsilon$ and $q> m > q^\varepsilon$,
we have
$$
N(k,m,q) \le   \binom{cq^{1/2}}{k} \binom{cq^{1/2}}{m} .
$$
\end{lem}

\begin{proof} We fix a set $\cV \subseteq \F_q$ of size 
$\#\cV = \fl{q^{\varepsilon/2}}$. We estimate the number 
$N(\cV, k,m,q)$ of sets
 $\cA, \cB \subseteq \F_q$ 
with $\# \cA = k$, $\cB=m$ such that 
$$\cQ = \cA + \cB\mand \cV \subseteq \cB.
$$ 

Let $\chi$ be the quadratic character. Let $\cU$ be the set of
elements $u \in \F_q$ such that for every $v \in \cV$ we have
$\chi(u+v) = 1$.  We see from Lemma~\ref{lem:Char 1} that $\# \cU\ll
q^{1/2}$.

Any set $\cA$ which contributes to $N(\cV, k,m,q)$ satisfies
$\cA \subseteq \cU$. Hence there are at most 
\begin{equation}
\label{eq:U}
\binom{\# \cU}{k} \le \binom{c_1q^{1/2}}{k}
\end{equation} 
possibilities for $\cA$ (where $c_1 > 0$ is some constant that 
depends only on $\varepsilon$).

Suppose now that $\cA$ is chosen.  Fixing an arbitrary set of $\fl{q^{\varepsilon/2}}$ elements
of $\cA$ and using the same argument we see that that 
the remaining elements of $\cB$ always belong to some 
fixed set  $\cW \subseteq \F_q$ of size 
$\#\cW \ll q^{1/2}$. Therefore,  there are at most 
\begin{equation}
\label{eq:W}
\binom{\# \cW}{k} \le \binom{c_2q^{1/2}}{m}
\end{equation} 
possibilities for the remaining elements of $\cB$ 
(where $c_2 > 0$ is some constant that 
depends only on $\varepsilon$).  Hence, combining~\eqref{eq:U}
and~\eqref{eq:W}, we obtain
$$
N(\cV, k,m,q) \le  \binom{c_1q^{1/2}}{k}\binom{c_2q^{1/2}}{m}.
$$
Summing over all choices for $\cV$ yields
\begin{equation*}
\begin{split}
N(k,m,q) & \le   \binom{q}{q^{\varepsilon/2}} \binom{c_1q^{1/2}}{k}\binom{c_2q^{1/2}}{m}\\
&\le q^{q^{\varepsilon/2}} \binom{c_1q^{1/2}}{k}\binom{c_2q^{1/2}}{m},
 \end{split}
\end{equation*}
which concludes the proof. 
\end{proof}

Now, using the fact that $N(k,m,q) \ne 0$ only if $q^{1/2}\ll k\ll q^{1/2}$
and $q^{1/2}\ll m\ll q^{1/2}$, 
see~\eqref{eq:low up}, we easily derive  Theorem~\ref{thm:Nq}
from~Lemma~\ref{lem:Nkmq}.

\section*{Acknowledgments}

During the preparation of the work, the second author was supported by
Russian Fund for Basic Research, Grant N.~14-01-00332,
and Program Supporting Leading Scientific Schools, Grant Nsh-3082.2014.1; 
the third author was supported by the 
Australian Research Council, Grant DP140100118.

\end{document}